\documentclass[english]{amsart}
\usepackage{amsmath,amssymb,amscd,amsfonts,mathrsfs,verbatim}
\usepackage[all]{xy}
\usepackage[mathcal]{euscript}
\usepackage{epsfig}
\usepackage{times}
\usepackage{hyperref}

\newcommand{\N}{{\mathbf N}}                   %Non-negative integers
\newcommand{\Z}{{\mathbf Z}}                   %Integers
\newcommand{\R}{{\mathbf R}}                   %Reals
\newcommand{\C}{{\mathbf C}}                   %Complex line
\newcommand{\CP}{\mathbf{P}^1}                %Riemann sphere 
\renewcommand{\H}{{\mathbf H}}                   %Hypercohomology
\newcommand{\Ct}{\widehat{\mathbf C}}          %Dual complex line
\newcommand{\CPt}{\widehat{\mathbf{P}}^1}     %Dual Riemann sphere 
\newcommand{\E}{{\mathcal E}}                  %Underlying holomorphic bundle
\newcommand{\F}{{\mathcal F}}                  %Modified holomorphic bundle
\newcommand{\G}{{\mathcal G}}                  %Modified holomorphic bundle
\newcommand{\Vt}{\widehat{V}}			%Dual smooth vector bundle
\newcommand{\Et}{\widehat{{\mathcal E}}}        %Dual holomorphic bundle
\newcommand{\Ft}{\widehat{F}}		       	%Dual filtration
\newcommand{\Pt}{\widehat{P}}        		%Dual parabolic set

\renewcommand{\tt}{\widehat{\theta}}                   %Dual Higgs field
\renewcommand{\O}{{\mathcal O}}                 %Holomorphic functions
\newcommand{\Nahm}{{\mathcal N}}              %Nahm transform map
\renewcommand{\d}{\mbox{d}}                      %Differential
\newcommand{\End}{{\mathcal E}nd}
\newcommand{\h}{{\mathfrak h}}
\newcommand{\e}{{\mathbf e}}			%Trivialisation of $\E$
\newcommand{\f}{{\mathbf f}}			%Trivialisation of $\F$
\newcommand{\g}{{\mathbf g}}			%Trivialisation of $\G$
\newcommand{\dbar}{\bar{\partial}}
\newcommand{\Dirac}{\partial\!\!\!\backslash}

\DeclareMathOperator{\res}{res}
\DeclareMathOperator{\Gl}{Gl}
\DeclareMathOperator{\Gr}{Gr}
\DeclareMathOperator{\diag}{diag}
\DeclareMathOperator{\Id}{Id}

\DeclareMathOperator{\im}{im}

\DeclareMathOperator{\rank}{rk}

\DeclareMathOperator{\red}{red}

\newtheorem{prop}{Proposition}[section]
\newtheorem{assn}[prop]{Assumption}
\newtheorem{rk}[prop]{Remark}
\newtheorem{clm}[prop]{Claim}

\newtheorem{defn}[prop]{Definition}

\newtheorem{thm}[prop]{Theorem}

\title[Nahm transformation]
{Nahm transformation for parabolic Higgs bundles on the projective line --- case of non-semisimple residues}
\author{Szil\'ard Szab\'o}

\begin{document}

\maketitle

\section{Introduction}

In this paper we define Nahm transformation for some singular %solutions of Hitchin's equations 
Higgs bundles on the complex projective line 
with finitely many %logarithmic singularities and one singularity of Poincar\'e rank (Katz-invariant) $1$. 
first-order poles and one second-order pole. 
Let $\C \subset \CP$ denote the complex affine and projective lines, endowed with the Euclidean metric, 
and with standard holomorphic coordinate denoted by $z\in \C$. 
We consider a parabolic harmonic bundle $(V, F_i^j , \dbar^{\E}, \theta , h)$ on $\CP$ with logarithmic singularities at some fixed points 
$z_1,\ldots,z_n\in\C$ and a second-order pole with semi-simple leading order term at infinity. 
Let $\Ct$ be a different copy of the complex affine line with coordinate $\zeta$, and $\CPt$ the associated projective line. 
The aim of this paper is to construct a transformed Higgs bundle $(\Vt, \dbar^{\Et} , \tt )$ 
%harmonic bundle $(\Vt, \dbar^{\Et} , \tt , \htr)$ 
on $\CPt$ and study the properties of the mapping 
\begin{equation}\label{eq:Nahm}
    \Nahm: (V, F_i^j, \dbar^{\E},\theta ) \mapsto (\Vt, \Ft_i^j, \dbar^{\Et} ,\tt ),
\end{equation}
%\begin{equation}\label{eq:Nahm}
%    \Nahm: (V, F_i^j, \dbar^{\E},\theta , h) \mapsto (\Vt, \Ft_i^j, \dbar^{\Et} ,\tt , \htr),
%\end{equation}
called Nahm transformation, on moduli spaces. 
In the case where the residues of $\theta$ at the singular points are semisimple, the transform was defined in \cite{Sz-these}, 
and its properties were further studied in %\cite{Sz-laplace}, 
\cite{A-Sz}, \cite{Sz-plancherel}. 
Therefore, in this note we will focus on the case where the residues of $\theta$ at the singular points are not necessarily semisimple.

\section{Construction of Nahm transform}\label{sec:Nahm}

In this section we define the parabolic Higgs bundle underlying the Nahm transform of a harmonic bundle on $\CP$, 
without going into the technical details of the constructions. 
In the later sections, we develop the technical tools necessary to make the construction rigorous, and sketch the proof of the results 
stated in this section. 

Let $C$ be a complex analytic curve. We denote by $\O_C$ and $K_C$ its structure sheaf and its canonical sheaf respectively, 
and by $\Omega^k$ the sheaf of locally $L^2$ differential $k$-forms on $C$. 

Let $V$ be a smooth vector bundle over $C$ of rank $r\geq 2$ and $\E$ be a holomorphic vector bundle with underlying smooth vector bundle $V$. 
The space of local sections of $\E$ may be conveniently described as the kernel of a partial differential operator $\dbar^{\E}$ 
of type $(0,1)$ on $V$. Let 
$$
  \theta : \E \to \E \otimes_{\O_C} K_C
$$
be a (possibly singular) morphism of $\O_C$-modules, called a Higgs field. 
The couple $(\E, \theta)$ is then called a Higgs bundle. 
Let $h$ be a smooth fibrewise Hermitian metric on $V$. Denote by $D_{Ch}$ the Chern connection associated 
to $\dbar^{\E}$ and $h$ and by $\theta^{\ast}$ the $h$-adjoint of $\theta$. 
Then $(V, \dbar^{\E},\theta , h)$ is called a harmonic bundle if and only if the connection 
\begin{equation}\label{eq:D}
   D = D_{Ch} + \theta + \theta^{\ast}
\end{equation}
is flat, i.e. its curvature $F_D$ vanishes. If this is the case, we let 
\begin{equation}\label{eq:nabla}
   \nabla = D^{1,0}
\end{equation}
denote the meromorphic integrable connection on the holomorphic vector bundle $E$ given by $D^{0,1}$. 

From now on, we let $(V, \dbar^{\E},\theta , h)$ denote a harmonic bundle over $\CP$ with some singularities. 
%In Section \ref{sec:L2} w
We will now spell out explicitly our assumptions on its singularities, as well as the definition of a compatible parabolic structure $F_i^j$. 
%Namely, $\theta$ is supposed to have logarithmic singularities at $z_1,\ldots ,z_n$ with arbitrary residues and 
%a second-order pole at infinity. We assume that the leading-order term of $\theta$ at infinity is semi-simple, 
%Consider the complex line $\C$ with its standard Euclidean metric $|\d z|^2$, and 
Fix finitely many distinct points $z_1,\ldots,z_n\in\C$. 
We consider the compactification $\CP$ of $\C$ by the point at infinity $z_0= [0:1]$. 
%Let $\E$ be a holomorphic vector bundle on $\CP$ of rank $r$. 
Let $\E$ be given the structure of a quasi-parabolic bundle on $\C$ with parabolic points 
$z_0,z_1,\ldots,z_n$, i.e. we assume that for every $i\in\{0,\ldots,n\}$ we are given a decreasing filtration 
of $\C$-vector subspaces of the fiber of $V$ at $z_i$
\begin{equation}\label{eq:parfiltr}
	\{ 0\} = F_i^{l_i} \subset F_i^{l_i-1} \subset \cdots \subset F_i^1 \subset F_i^0 = V_{z_i}
\end{equation}
of some length $1\leq l_i\leq r$. 
For $i\in\{0,\ldots,n\},j\in\{ 0,\ldots,l_i-1\}$ consider the graded vector spaces associated to (\ref{eq:parfiltr}) 
\begin{equation}\label{eq:pargr}
	\Gr_{i}^j = \Gr_{F_i}^j = F_i^j/F_i^{j+1}.
\end{equation}
We fix parabolic weights $\{\alpha_i^j\}$ for $i\in\{0,\ldots,n\}, j\in\{ 0,\ldots,l_i-1\}$ satisfying 
\begin{equation}\label{eq:parweights}
	1> \alpha_i^{l_i-1} > \cdots > \alpha_i^0 \geq 0. 
\end{equation}
For every $0 \leq i \leq n$ we will take a local holomorphic trivialisation $\{ \e_i^s\}_{s=1}^r$ of $\E$ near $z_i$ 
compatible with the filtration $F_i^j$ in the sense that $F_i^j$ is spanned by the evaluations at $z_i$ of the vectors 
$$
  \e_i^1, \ldots, \e_i^{\dim F_i^j}. 
$$
With respect to such a compatible basis, we will use the diagonal matrix 
$$
  \diag(\alpha_i^j )_{j=0}^{l_i-1}
$$
consisting of the parabolic weights, each $\alpha_i^j$ repeated with multiplicity equal to $\dim \Gr_{i}^j$. 
We assume that 
$$
	\theta\in \Gamma(\CP, \End(\E)\otimes K_{\CP}(\log(z_1)+\cdots +\log(z_n) + 2\cdot z_0))
$$ 
is a Higgs field on $\E$ with logarithmic singularities at $z_1,\ldots,z_n$ and a second-order pole 
with semi-simple leading order term at infinity, compatible with the parabolic structure. 
We will call such a Higgs field singular. 
By compatibility in the logarithmic case we mean that the residue 
\begin{equation}\label{eq:residue}
	\res_{z_i}(\theta) = \theta%\righthalfcup 
	((z-z_i)\partial_z)
\end{equation}
of $\theta$ at $z_i$ preserves the filtration $F_i^{\bullet}$: 
\begin{equation}\label{eq:logHiggsfield}
	\res_{z_i}(\theta) : F_i^j \to F_i^j
\end{equation}
for every $i\in\{1,\ldots,n\}, j\in\{ 0,\ldots,l_i-1\}$ . 
For the second-order pole $z_0$ at infinity, we require an equality 
\begin{equation}\label{eq:Higgsinfinity}
	\theta = \frac A2 \d z + B \frac{\d z}{z} + \mbox{lower order terms},
\end{equation}
in some local trivialisation of $\E$, where $A$ is a semi-simple $r\times r$ matrix and 
$B$ any $r\times r$ matrix, both preserving the image of the filtration (\ref{eq:parfiltr}) under the isomorphism 
$\E|_{\infty}\cong \C^r$ given by the trivialisation. 
We denote by $\Pt\subset \Ct$ the eigenvalues of $A$. 
Moreover, let us denote by $H$ the centraliser of $A$ in $\Gl_r(\C)$  and by $\h$ its Lie-algebra. 
Then up to applying a holomorphic gauge transformation near $\infty$ we can arrange that $B\in\h$; 
in what follows we will therefore assume $B\in\h$. 
\begin{clm}\label{clm:A-preserves-F}
 $A$ preserves the filtration (\ref{eq:parfiltr}) if and only if for any $l$ the piece $F_0^l$ in 
(\ref{eq:parfiltr}) is the direct sum of its intersections with the various eigenspaces of $A$. 
\end{clm}
\begin{proof}
The direction $\Leftarrow$ is trivial. For the converse, if for instance 
$$
  \lambda_1 v_1 + \lambda_2 v_2 \in F_0^l
$$
for some $\lambda_j\neq 0$ and $v_j$ in the $\zeta_j$-eigenspace of $A$ with $\zeta_1\neq \zeta_2$ then 
$$
  A(\lambda_1 v_1 + \lambda_2 v_2) = \zeta_1(\lambda_1 v_1 + \lambda_2 v_2) + (\zeta_2 - \zeta_1)\lambda_2 v_2.
$$
Now as by assumption the left-hand side and the first term on the right-hand side belong to $F_0^l$, the 
same thing follows for the second term on the right-hand side, and thus (as $\zeta_2 - \zeta_1\neq 0$) for $\lambda_2 v_2$ too, 
which in turn implies the same thing for $\lambda_1 v_1$ as well. 
The same kind of argument applies for a vector with components in more than just two different eigenspaces.
\end{proof}

By compatibility, $\res_{z_i}(\theta)$ acts on the spaces (\ref{eq:pargr}). Let us denote by $\res_{z_i}(\theta)^j$ this action and let 
$$
	\res_{z_i}(\theta)^j = S_i^j  + N_i^j 
$$
be its decomposition into its semi-simple and nilpotent components respectively. 
We may (and henceforth will) assume that the compatible trivialisations $\{ \e_i^s\}_{s=1}^r$ are chosen so that $S_i^j$ are diagonal for each $i,j$. 
The generalized eigenspaces of $S_i^j$ then define a block-decomposition of $\Gr_i^j$. 
To each such block there corresponds a single eigenvalue of $\res_{z_i}(\theta)^j$, and the eigenvalues are different on different blocks. 
\begin{clm}\label{clm:-theta-block-diagonal}
The sections $\e_i^s$ can be chosen so that for each $j$ the restriction of $\theta$ to the subbundle of $\E$ spanned by the vectors 
$\e_i^s$ with $j(s) = j$ is block-diagonal with respect to the block-decomposition defined by $S_i^j$. 
\end{clm}
\begin{proof}
 The bundle $\E$ splits holomorphically as a direct sum of vector subbundles corresponding to various eigenvalues of $\theta$. 
 By compatibility of $\res_{z_i}(\theta)$ with $F_i^{\bullet}$, each such direct summand is a direct sum of its graded pieces for $F_i^j$, 
 see Claim \ref{clm:A-preserves-F}. The union of a local holomorphic trivialisation for the graded pieces of the direct summands for all 
 possible choices fulfills the desired property. 
\end{proof}

By the Jordan-H\"older theorem there is an increasing filtration $W_{i,\bullet}^j$ of $\Gr_{F_i}^j$ associated to $N_i^j$ satisfying 
\begin{enumerate}
\item for all $k\in\Z$, $N_i^j$ maps $W_{i,k}^j$ into $W_{i,k-2}^j$
\item for all $k\in\N$ the map $(N_{i}^j)^k$ is an isomorphism $\Gr^{W_i^j}_k\Gr_{F_i}^j\to\Gr^{W_i^j}_{-k}\Gr_{F_i}^j$ 
(in the notation $(N_{i}^j)^k$ the index $j$ refers to restriction of $N_i$ to $\Gr_{F_i}^j$, whereas the upper index $k$ stands for 
$k$-fold composition of $N_i^j$ with itself). 
\end{enumerate} 
The filtration $W_{i,\bullet}^j$ is called a weight filtration; it is not unique, however its length and the dimensions of its graded pieces 
$$
  \Gr_{i,k}^j = \Gr^{W_i^j}_k\Gr_{F_i}^j
$$
are unique. Observe that in the case $i=0$, the assumption $B\in\h$ implies $N_{0,j}\in\h$ for every $j$.

A Hermitian metric $h$ in $\E$ in some neighborhood of $z_i$ ($i\in\{1,\ldots,n\}$) is said to be compatible with 
$\theta$ if and only if it is mutually bounded with the diagonal metric 
\begin{equation}\label{eq:modelmetric}
	h_0 = \diag (|z-z_i|^{2\alpha_i^{j(s)}} (-\log |z-z_i|)^{k(s)})_{s=1}^r
\end{equation}
with respect to some (or equivalently, any) holomorphic trivialisation $\{ \e_i^s\}_{s=1}^r$ compatible with the filtrations $F$. 
Here $j(s)$ refers to the largest $j\in\{ 0,\ldots,l_i-1\}$ such that $\e_i^s\in F_i^j$ and 
$k(s)$ refers to the smallest $k\in\Z$ such that $\e_i^s\in W_{i,k}^{j(s)}$.  
For $i=0$, we require an analogous behaviour, with $z^{-1}$ replacing the local coordinate $z-z_i$: 
\begin{equation}\label{eq:modelmetricinf}
	h_0 = \diag (|z|^{-2\alpha_0^j(s)} (\log |z|)^{k(s)})_{s=1}^r.
\end{equation}
For a compatible harmonic metric, for all $0 \leq i \leq n$ the diagonal matrix consisting of the parabolic weights of $E = \ker (D^{0,1})$ is given by 
\begin{equation}\label{parabolic-relationship}
 \diag(\beta_i^s )_{s=1}^r = \diag(\alpha_i^j - 2 \Re (S_i^j ) )_{j=1}^{l_i-1} ,
\end{equation}
where the arguments of $\diag$ on the right-hand side are diagonal matrices of dimension $\dim \Gr_i^j$ each. 
Observe that $\beta_i^s$ does not depend on $s$ on any eigenspace of $S_i^j$. 
The singularity of $\nabla = D^{1,0}$ at $z_i$ for all $1 \leq i \leq n$ is logarithmic, and at infinity $\nabla$ has a singularity with Katz-invariant $1$. 
The second-order term of $\nabla$ at infinity is simply $A$ by the results of \cite{Biq-Boa}, 
that is to say twice the second-order term of $\theta$: 
\begin{equation}\label{eq:Dinfinity}
 D = \d + A \d z + O(z^{-1}) \d z, 
\end{equation}
with respect to some holomorphic trivialization of $E$. 
Finally, the relationship between the graded pieces of the residue of the Higgs filed and that of the integrable connection for the filtration $F_i$ is 
\begin{equation}\label{eq:residue-relationship}
 \res_{z_i}(\theta ) = \frac{\res_{z_i}(\nabla ) - \diag(\beta_i^s )_{s=1}^r}{2}, 
\end{equation}
c.f. the table in the Synopsis of \cite{Sim}.

From now on we let $(V, F_i^j, \dbar^{\E},\theta , h)$ denote a harmonic bundle on $\CP$ with parabolic structure and admissible harmonic metric, 
and singularity behaviour fixed as above. We now make important assumptions necessary to carry out our construction. 
\begin{assn}\label{assn:main}
 For any $i,j$, 
\begin{enumerate}
\item \label{assn:main0}
 for $i = 0$, $0$ is not an eigenvalue of $\res_{z_0}(\theta)^j$; 
 \item \label{assn:main1} 
 for $i > 0$ if $\alpha_i^j=0$ then the nilpotent part $N_i^j$ of the endomorphism $\res_{z_i}(\theta)^j$ acts trivially on the $0$-eigenspace 
 of $\res_{z_i}(\theta)^j$; 
 \item \label{assn:main2} 
 for $i > 0$ if $\alpha_i^j > 0$ then $0$ is not an eigenvalue of $\res_{z_i}(\theta)^j$.
% \item \label{assn:main4} 
% $(V, D)$ is without flat factors: for every $\zeta\in\C$ it has no flat subbundle of the form $(\O_{\CP}, \d + \zeta \d z)$. 
\end{enumerate}
\end{assn}

%\begin{rk}
 
%\end{rk}

Let us note a straightforward consequence. 

\begin{clm}\label{assn:main3} 
 If the condition (\ref{assn:main0}) above holds then $(\E, \theta)$ has no Higgs subbundle of the form $(\O_{\CP} (m), \zeta \d z)$ 
 for any $\zeta\in\C$. 
\end{clm}
\begin{rk}
 This property is an analogue of an instanton being without flat factors, c.f. Definition 3.2.2 of \cite{DK}.  
 Notice that by (\ref{eq:Higgsinfinity}) such a flat factor could only exist for $\zeta$ an eigenvalue of $A/2$. 
\end{rk}

\begin{proof}
 This is immediate, as the eigenvalue of the residue of $\zeta \d z$ at infinity is $0$. 
\end{proof}

Now, given $\zeta \in \Ct \setminus \Pt$ we define a twisted flat connection on $\CP$ by 
\begin{equation}\label{Dzeta}
  D_{\zeta} = D - \zeta \d z. 
\end{equation}
We consider the twisted elliptic complex 
\begin{equation}\label{eq:elliptic-complex-Dzeta}
 	0\to V\xrightarrow{D_{\zeta}}  V\otimes \Omega^1 \xrightarrow{D_{\zeta}}  V\otimes \Omega^2 \to 0. 
\end{equation}
It easily follows from the form (\ref{eq:Dinfinity}) that for any $\zeta_j \in \Pt$ the flat sections of $V$ for $D_{\zeta}$ have exponential behaviour 
$$
  \exp (-(\zeta_j - \zeta)z + P(\log |z|))
$$
on the $\zeta_j$-eigenspace of $A$ for some function $P$ of at most polynomial growth. In particular, this complex has trivial $L^2$-cohomology in degree $0$. 
By a duality argument, the same then follows for its $L^2$-cohomology in degree $2$ too. 
In particular, it follows from continuity of the index that the dimension of its first $L^2$-cohomology space is independent of $\zeta \in \Ct \setminus \Pt$. 
We then define the fiber $\Vt_{\zeta}$ of the transformed vector bundle $\Vt$ as the first $L^2$-cohomology space of (\ref{eq:elliptic-complex-Dzeta}). 
In Section \ref{sec:Fredholm} we will show that $\Vt_{\zeta}$ has an equivalent description as the kernel of the twisted Laplace-operator 
$$
  \Delta_{\zeta} = D_{\zeta} D_{\zeta}^* + D_{\zeta}^* D_{\zeta}
$$
on its $L^2$-domain. 
By Kodaira and Spencer's Fundamental Theorem \cite{KS}, 
%By the Banach-space implicit function theorem, $\Vt_{\zeta}$ is then a smoothly varying subspace of some Banach space, hence 
there exists a smooth vector bundle $\Vt$ over $\Ct \setminus \Pt$ with fiber over $\zeta$ given by $\Vt_{\zeta}$. 
In Section \ref{sec:Dolbeault} we show that $\Vt_{\zeta}$ is isomorphic to the first hypercohomology space of a twisted Dolbeault complex. 
As a consequence, we obtain a holomorphic bundle $\Et$ over $\Ct \setminus \Pt$ with underlying smooth vector bundle $\Vt$ and a meromorphic Higgs field $\tt$ on $\Et$. 
Furthermore, as we explain in Section \ref{sec:par}, this identification provides us with a natural extension of $\Et$ to $\CPt$ and 
even a transformed parabolic structure $\Ft_i^j$. 
%In Section \ref{sec:st} we discuss the generic singularity behaviour of the transformed Higgs bundle. 
%In Section \ref{sec:dR} we go on to define a transformed Hermitian metric $\htr$ on $\Vt$ and show that $(\Vt,  \Ft_i^j, \dbar^{\Et} ,\tt , \htr)$ 
%is a parabolic harmonic bundle. We achieve this by identifying $\Vt_{\zeta}$ to the first hypercohomology space of a twisted de Rham complex. 
%We end the paper by showing that the transform (\ref{eq:Nahm}) is a hyperK\"ahler isometry between suitable moduli spaces. 
The main result of the paper is Theorem \ref{thm:main}, stating that the transformed Higgs field is meromorphic, it is compatible 
with the parabolic structure $\Ft_{i}^{\bullet}$, and that the transform (\ref{eq:Nahm}) respects the Dolbeault complex structures of the 
moduli spaces.

\section{Fredholm theory}\label{sec:Fredholm}

Much of the analysis has been carried out in Chapter 2 of \cite{Sz-these}, so here we will only indicate the differences in the argument that one needs in order to take into 
account non-zero nilpotent parts of the residues. In particular, as we assume that the leading-order term at infinity is semi-simple, the local analysis near infinity
is exactly the same as in \cite{Sz-these}. The same holds for the global Hodge-theoretic arguments. 
Therefore, we only need to work locally near a logarithmic point $z_i$ for a fixed $1\leq i \leq n$. 

Let $\varepsilon >0$ be a small real number and consider the disc $B = B_{\varepsilon}(z_i)$ of radius $\varepsilon$ centered at $z_i$. 
Let us define the local norm 
$$
  H^1(\Omega^k) = \left\{ f\in \Omega^k(B)\otimes V | \: \int_B |f|^2 + |D_{Ch}(f)|^2 + |\theta (f)|^2 < \infty \right\}, 
$$
where we use the Euclidean metric on $B$ and the metric $h$ on  $V$ to compute the norms involved. 
Then we have the following analog of Claim 2.4 of \cite{Sz-these}. 
\begin{clm}\label{clm:H1}
 The space $H^1(\Omega^k)$ does not depend on the specific harmonic bundle with fixed singularity behaviour. 
\end{clm}
\begin{proof}
Fix a unitary frame $\sigma_i^s$ ($1\leq s\leq r$) of $V$ near $z_i$ compatible with the parabolic filtration. 
%(Notice that this differs from the  holomorphic trivialisation $\{ \e_i^s\}_{s=1}^r$ chosen in Section \ref{sec:Nahm}.) 
The statement clearly holds on the regular part of $V$ spanned by the vectors $\sigma_i^s$ where $\theta$ and $D$ are both regular and the parabolic weights vanish: 
indeed, on this component of $V$ the space $H^1$ is simply the usual Sobolev space $L^{2,1}$. 
Let now $(r, \varphi)$ be the polar coordinates on $B$. 
According to Theorem 1 of \cite{Sim}, the matrix forms of $D_{Ch}, \theta$ with respect to this trivialisation differ from some model 
by an endomorphism-valued one-form $\tau$ of growth order $O(r^{-1}|\log r|^{-1})$. The model reads as 
\begin{align*}
 D_{Ch, i} & = \d + \sqrt{-1} \Re (\res_{z_i}(\nabla ))\d \varphi \\
 \theta_{i} & = \frac{\res_{z_i}(\nabla ) - \diag(\beta_i^s )_{s=1}^r}{2 (z-z_i)} \d z.
\end{align*}
(See e.g. \cite{Sz-these} (1.18)-(1.19).) Furthermore, we may assume that $\res_{z_i}(\nabla )$ is upper triangular. 
Denote by $\mu_i^s$ the eigenvalue of $\res_{z_i}(\nabla )$ corresponding to the basis element $\sigma_i^s$ and by $\lambda_i^s$ the corresponding eigenvalue of $\res_{z_i}(\theta )$.

We need to prove that the norm $H^1$ defined above is equivalent to the analogous norm $H^1_{i}(\Omega^k)$ obtained from this model. 
For this it is sufficient to show that for any $f\in \Omega^k(B)$ the norm of the image $\tau(f \sigma_i^s )$ under the perturbation term is bounded above by a suitable constant multiple of 
$|D_{Ch, i} (f \sigma_i^s)|$ or by a suitable constant multiple of $|\theta_{i}(f \sigma_i^s)|$. 

By our first observation, it is sufficient to focus on vectors $\sigma_i^s$ such that either $\beta_i^s \neq 0$ or $\sigma_i^s$ is not contained in the kernel of at least one of 
the two endomorphisms $\res_{z_i}(\nabla ), \res_{z_i}(\theta )$. 
Now, the key point is that by Assumption \ref{assn:main} (\ref{assn:main1}) if $\mu_i^s = 0 = \beta_i^s$ then the nilpotent part acts trivially on $\sigma_i^s$ and we are back to the regular case, 
already discussed above. Therefore, we have the following three cases. 

First, if $\Im (\mu_i^s ) \neq 0$ then there exists some $c>0$ only depending on $\varepsilon$ and $\mu_i^s$ such that for any $f\in \Omega^k(B)$ we have 
$$
  | \theta_i(f \sigma_i^s) |^2  \geq c |r^{-1}|\log r|^{-1} f |^2
$$
over $B$. Hence, in this case some fixed multiple of $| \theta_i(f \sigma_i^s) |$ bounds the norm of the perturbation term. 

Second, if $\Im (\mu_i^s ) = 0$ but $\Re (\mu_i^s ) \neq 0$, then by standard Fourier theory on the cylinder there exists some $c>0$ only depending on $\varepsilon$ and $\mu_i^s$ such that 
$$
  |D_{Ch,i}(f \sigma_i^s)|^2 > c |r^{-1}|\log r|^{-1} f |^2 
$$
over $B$. Hence, in this case some fixed multiple of $| D_{Ch,i}(f \sigma_i^s) |$ bounds the norm of the perturbation term. 

Third, if $\mu_i^s = 0$ but $\beta_i^s \neq 0$ then just as in the first case, some fixed multiple of $| \theta_i(f \sigma_i^s) |$ bounds the norm of the perturbation term from above. 
This finishes the proof. 
\end{proof}

Now define the Dirac operator 
$$
  \Dirac_{\zeta} = D_{\zeta} - D_{\zeta}^* : (\Omega^0 \oplus \Omega^2) \otimes V \to \Omega^1 \otimes V 
$$
where $D_{\zeta}^*$ stands for the $h$-adjoint of $D_{\zeta}$. Then we have the analog of Theorem 2.6 of \cite{Sz-these}.
\begin{prop}
 The operator $\Dirac$ is Fredholm from $H^1$ to $L^2$. 
\end{prop}

\begin{proof}
 Just as in Section 2.2 \cite{Sz-these}, it is possible to glue a parametrix of $\Dirac$ over the complement of a neighborhood of the parabolic points and exact 
 inverses of the local model Dirac operators $\Dirac_i$. The key point is that $0$ is not a critical weight for the translation-invariant model operator $r_i \Dirac_i$, 
 where $r_i = |z-z_i|$. This follows exactly as in Subsection 2.2.1 {\it loc. cit.}: $0$ is a critical weight for the action of $\Dirac_i$ on $\sigma_i^s$ if and only if 
 $\mu_i^s = 0 = \beta_i^s$. However, by Assumption \ref{assn:main} (\ref{assn:main1}) this implies that $\Dirac_i$ is the Dirac operator corresponding to a regular connection and 
 non-singular metric. In this latter case, by Claim \ref{clm:H1} the function space $H^1$ is usual Sobolev space $L^{2,1}$ and the statement is classical. 
\end{proof}

We define the Dirac--Laplace operator 
$$
  \Delta_{\zeta} = - D_{\zeta} D_{\zeta}^* - D_{\zeta}^* D_{\zeta} : \Omega^1 \otimes V \to \Omega^1 \otimes V. 
$$
The proof of Theorems 2.16 and 2.21 of \cite{Sz-these} then imply the following. 
\begin{prop}
 The first $L^2$-cohomology of the elliptic complex (\ref{eq:elliptic-complex-Dzeta}) is canonically isomorphic to the cokernel of the Dirac-operator $\Dirac_{\zeta}$ defined on $H^1$ 
 and to the kernel of the Dirac--Laplace operator defined on its $L^2$ domain. 
\end{prop}

\section{Dolbeault interpretation}\label{sec:Dolbeault}

%\section{Dolbeault interpretation}

In this section we write down and analyze an explicit Dolbeault complex admitting an $L^2$-resolution 
for a model Hermitian metric on the fibers and Euclidean metric on the base. 
Our analysis is very similar to %(but somewhat more precise than) 
that of Mochizuki (c.f. Section 5.1 of \cite{Moc}) 
for irregular $\lambda$-connections for the Poincar\'e metric. 
The results in this section also hold for smooth projective curves of any genus, but for 
the sake of simplicity of the exposition we limit ourselves to the case of the projective line. 
In this section we again work with the compatible local holomorphic trivialisations $\e_i^s$ of $\E$ near $z_i$ as in Section \ref{sec:Nahm}. 

As a preliminary observation, let us give a slightly different description of the bundle $\Vt$. 
Namely, the unitary gauge transformation
\begin{equation}\label{eq:gauge-tr}
   \exp \left( \frac 12 (\bar{\zeta} \bar{z} - \zeta z) \right)
\end{equation}
transforms $D_{\zeta}$ into the deformation 
\begin{equation}\label{eq:Dzeta-2}
  D - \frac{\zeta}2 \d z - \frac{\bar{\zeta}}2 \d \bar{z}.
\end{equation}
In particular, this is a self-adjoint deformation, hence it corresponds to 
\begin{align*}
 \dbar^{\E}_{\zeta} & = \dbar^{\E} \\
 \theta_{\zeta} & = \theta - \frac{\zeta}2 \d z .
\end{align*}
The upshot is that the holomorphic bundle $\E$ is preserved under this deformation. 
Now, as (\ref{eq:gauge-tr}) is unitary, it preserves the spaces of $L^2$-sections of $\Omega^k$, and therefore 
induces isomorphisms on the $L^2$-cohomology spaces of the two deformations (\ref{Dzeta}) and (\ref{eq:Dzeta-2}). 
We infer 
$$
  \Vt_{\zeta} = H^1_{L^2} \left( D - \frac{\zeta}2 \d z - \frac{\bar{\zeta}}2 \d \bar{z} \right).
$$
In this section, we will work with this latter deformation, as it is more convenient to study the transform of the underlying 
Higgs bundle of the harmonic bundles. 

Now, let us start the analysis of the $L^2$-cohomology. 
First, one can easily check that in the case $i>0$ we have $\e_i^s\in L^2(\E,h,|\d z|^2)$ for all $s\in\{1,\ldots,r\}$; 
moreover $(z-z_i)^{-1}\e_i^s\in L^2(\E,h,|\d z|^2)$ holds if and only if either 
\begin{enumerate}
\item $\alpha_i^{j(s)}>0$ or
\item $\alpha_i^{j(s)}=0$ and $k(s) < -1$
\end{enumerate}
is fulfilled; finally, $(z-z_i)^{-2}\e_i^s\notin L^2(\E,h,|\d z|^2)$ holds for all $s\in\{1,\ldots,r\}$. 
On the other hand, for $i=0$ we have $\e_0^s\notin L^2(\E,h,|\d z|^2)$ for all $s\in\{1,\ldots,r\}$; 
moreover $z^{-1}\e_0^s\in L^2(\E,h,|\d z|^2)$ holds if and only if either 
\begin{enumerate}
\item $\alpha_0^{j(s)}>0$ or
\item $\alpha_i^{j(s)}=0$ and $k(s) < -1$
\end{enumerate}
is fulfilled; finally, $z^{-2}\e_0^s\in L^2(\E,h,|\d z|^2)$ holds for all $s\in\{1,\ldots,r\}$. 

We now proceed to identify the Dolbeault complex 
\begin{equation}\label{eq:Dolbcopml}
	\F \xrightarrow{\theta} \G \otimes K_{\CP}(2\cdot z_0 + z_1 + \cdots + z_n)
\end{equation}
that admits an $L^2$ Dolbeault resolution for the Euclidean metric. 
The sheaves $\F,\G$ are characterised as certain elementary transforms of the sheaf of local sections of $\E$ 
at $z_0,z_1,\ldots,z_n$ along some subspaces of the fiber $\E_{z_i}$, so that they are equal to $\E$ away from the 
punctures $z_i$ for all $i\geq 0$. 
Let us start by defining a local frame $\{\g_i^s\}_{s=1}^r$ of $\G$ near $z_i$ for $i>0$: 
\begin{enumerate}
\item Case $\alpha_i^{j(s)}=0$
\begin{enumerate}
	\item sub-case $k(s) < -1$: set $\g_i^s=\e_i^s$
	\item sub-case $k(s)\geq -1$: set $\g_i^s=(z-z_i)\e_i^s$
\end{enumerate}
\item Case $\alpha_i^{j(s)}> 0$: set $\g_i^s=\e_i^s$. 
\end{enumerate}
For $i=0$ the frame is
\begin{enumerate}
\item Case $\alpha_0^{j(s)}=0$
\begin{enumerate}
	\item sub-case $k(s) < -1$: set $\g_0^s=z^{-1}\e_0^s$
	\item sub-case $k(s)\geq -1$: set $\g_0^s=z^{-2}\e_0^s$
\end{enumerate}
\item Case $\alpha_i^{j(s)}> 0$: set $\g_0^s=z^{-1}\e_0^s$. 
\end{enumerate}

Let us denote by $\lambda_i^{s}$ the eigenvalue of $S_i^{j(s)}$ corresponding to the vector $\e_i^s$. 
For $i>0$ we construct an explicit frame $\{\f_i^s\}_{s=1}^r$ of $\F$ starting out of $\{ \e_i^s\}_{s=1}^r$, 
depending on $k(s)$ and whether $\alpha_i^{j(s)}$ and $\lambda_i^{s}$ vanish or not. 
Here is the definition of the local frame of $\F$ in the case $i>0$: 
\begin{enumerate}
\item Case $\alpha_i^{j(s)}=0$
\begin{enumerate}
%	\item sub-case $\lambda_i^{s}=0, N_{i,j(s)}(e_i^s(z_i))=0, k(s)<-1$: set $\f_i^s = (z-z_i)^{-1}\e_i^s$
	\item sub-case $\lambda_i^{s}=0$ (and as a consequence necessarily $N_i^{j(s)}(e_i^s(z_i))=0%, k(s)\geq -1
											$ by Assumption \ref{assn:main} (\ref{assn:main1})): set $\f_i^s = \e_i^s$
%	\item sub-case $\lambda_i^{s}=0, N_i^{j(s)}(e_i^s(z_i))\neq 0, k(s)<1$: set $\f_i^s = \e_i^s$
%	\item sub-case $\lambda_i^{s}=0, N_i^{j(s)}(e_i^s(z_i))\neq 0, k(s)\geq 1$: set $\f_i^s = (z-z_i)\e_i^s$
	\item sub-case $\lambda_i^{s}\neq 0, k(s)<-1$: set $\f_i^s = \e_i^s$
	\item sub-case $\lambda_i^{s}\neq 0, k(s)\geq -1$: set $\f_i^s = (z-z_i)\e_i^s$
\end{enumerate}
\item Case $\alpha_i^{j(s)}>0$: set $\f_i^s = \e_i^s$. 
%\begin{enumerate}
%	\item sub-case $\lambda_i^{s}=0, N_i^{j(s)}(e_i^s(z_i))=0$: set $\f_i^s = %(z-z_i)^{-1} %%%%%%% NEM LEHET: HA A FILTRALASBAN BAL-ALSO RESZEBEN THETANAK NEM 0 AZ $O(1)$ ELEME (AMI GENERIKUSAN IGAZ), AKKOR A KEPE NEM $L^2$
  %\e_i^s$
%	\item sub-case $\lambda_i^{s}=0, N_i^{j(s)}(e_i^s(z_i))\neq 0$: set $\f_i^s = \e_i^s$
%	\item sub-case $\lambda_i^{s}\neq 0$: set $\f_i^s = \e_i^s$. 
%\end{enumerate}
\end{enumerate}
For $i=0$ the sheaf of local sections of $\F$ doesn't depend on the eigenvalues of the polar part of $\theta$, as for the Euclidean metric 
we have $|\theta \e |_{h,|d z|^2}\leq K | \e |_{h,|d z|^2}$ for any section $\e$ of $\E$ near $\infty$ and a constant $K>0$ 
only depending on the leading order term $A$ in (\ref{eq:Higgsinfinity}). 
Therefore, a local frame of $\F$ in the case $i=0$ is given by $\f_0^s = \g_0^s$ for $1\leq s \leq r$. 

One can check that the definitions of $\F$ and of $\G$ above are independent of the choice of a compatible local frame $\{\e_i^s\}$. 
In addition, $\F$ is the lower elementary transformation of $\G(z_1+\cdots + z_n)$ at the points $z_i$ along a 
subspace $W\subset \G_{|_{z_i}}$ for $i\in\{1,\ldots,n\}$: 
$$
	0\to \F \to \G(z_1+\cdots + z_n) \to W \to 0; 
$$
in particular, $\F$ is naturally a subsheaf of $\G(z_1+\cdots + z_n)$. 

Now, let us set $D'' = \bar{\partial}^{\E}+ \theta$. 
Consider the twisted Higgs bundle 
\begin{equation}\label{eq:twistedHiggs}
	\theta_{\zeta} = \theta - \frac{\zeta}{2}\Id_{\E}\d z
\end{equation}
and the twisted $D''$-operator 
$$
	D''_{\zeta} = \bar{\partial}^{\E}+ \theta_{\zeta}. 
$$
Notice that as $\F\hookrightarrow\G(z_1+\cdots + z_n)$, the morphism $\theta_{\zeta}$ maps the sheaf $\F$ to 
$$
  \G(z_1+\cdots + z_n)\cong \G\otimes K_{\CP}(2\cdot z_0 + z_1+\cdots + z_n). 
$$

Just as in Proposition 4.13 \cite{Sz-these}, we have the following. 
\begin{prop}\label{prop:L2Dol}
For every $\zeta\in \Ct \setminus \Pt$, the fiber $\Vt_{\zeta}$ is isomorphic to the first $L^2$-cohomology space of the complex
$$
	0\to V\xrightarrow{D''_{\zeta}}  V\otimes \Omega^1 \xrightarrow{D''_{\zeta}}  V\otimes \Omega^2 \to 0, 
$$
and to the first hypercohomology space of the twisted holomorphic Dolbeault complex 
\begin{equation}\label{eq:twistedDolbcopml}
	\F \xrightarrow{\theta_{\zeta}} \G\otimes K_{\CP}(2\cdot z_0 + z_1+\cdots + z_n).
\end{equation}
\end{prop}

\begin{proof}
The first statement follows as both $L^2$-cohomology spaces are isomorphic to spaces of harmonic sections for the Laplace-operators associated to (\ref{eq:Dzeta-2}) and to 
$D''_{\zeta}$ respectively; however, by the Weitzenb\"ock formula, these Laplace-operators agree up to a constant. 

Let us come to the identification in terms of the first hypercohomology space. 
By construction, near the parabolic points the sections of $\F$ are the local meromorphic sections 
$\e$ of $\E$ such that $\theta(\e)$ is also $L^2$, or equivalently, the local $L^2$ sections 
$v$ of $V$ such that $\bar{\partial}^{\E}(v)=0$ and $\theta(v)\in L^2$. 
Notice that for any given $\zeta\in\Ct$ and $v\in L^2$ we clearly also have $\frac{\zeta}{2}v\in L^2$, 
so the same statement holds for $\theta$ replaced by $\theta_{\zeta}$ too. 
By the above analysis, $\theta_{\zeta}$ induces a morphism of $L^2$-resolutions of the sheaves of the complex. 
In particular, the resolutions involved are acyclic, and this then proves the statement. 
%A standard double complex argument then gives the proof. 
\end{proof}

%\begin{rk}
%For $\alpha_i^{j(s)}>0, \lambda_i^{s} =0, N_{i,j(s)}(\e_i^s(z_i))=0$ %(case (1a) of the definition of $\F$) 
%(i.e., $\res_{z_i}(\theta)(\e_i^s(z_i))=0$) and in the non-generic case where we have in addition 
%that the $k(s)\geq 1$ components of $\theta(\e_i^s)$ vanish at $z_i$, then even the sections 
%$(z-z_i)^{-1}\e_i^s$ lie in $L^2$, along with their image by $\theta$......... 
%Similarly for the case (2a) of the definition of $\F$.........
%\end{rk}

\section{Extension and transformed parabolic structure}\label{sec:par}

The analytic construction of $\Vt$ is only defined over $\Ct \setminus \Pt$. 
In this section, we will use the results of Section \ref{sec:Dolbeault} to extend it over $\CPt$. 
Actually, we will extend the holomorphic vector bundle $\Et$ holomorphically to $\CPt$. 
Furthermore, we will also define a parabolic structure on this extension. 

We start by defining the transformed holomorphic vector bundle $\Et$ on $\Ct$: this is a simple consequence of the identification of $\Vt_{\zeta}$ 
as a first hypercohomology space (Proposition \ref{prop:L2Dol}), because the morphism $\theta_{\zeta}$ depends holomorphically 
(actually even algebraically) on $\zeta$. 
%Alternatively, as $D''_{\zeta}$ depends holomorphically on $\zeta$, holomorphicity of the space of solutions also follows from Theorem 5 \cite{KS}. 

Now we turn to the extension of $\E$ over $\infty \in \CPt$. 
For this purpose, we let $s_0, s_{\infty} \in \O_{\CPt}(1)$ denote the sections such that on the affine chart $\C \subset \CPt$ we have 
$$
  s_0 (\zeta ) = \zeta, \quad s_{\infty} (\zeta ) = 1. 
$$
We now modify (\ref{eq:twistedDolbcopml}) into 
\begin{equation}\label{eq:theta-extension}
   \theta_{\zeta} = \theta \otimes s_{\infty} - \frac{\zeta}{2}\Id_{\E}\d z \otimes s_0 : \pi_1^*\F \to \pi_1^*\G\otimes K_{\CP}(2\cdot z_0 + z_1+\cdots + z_n) \otimes \pi_2^*  \O_{\CPt}(1)
\end{equation}
with $\pi_i$ the projection from $ \CP \times \CPt$ to its $i$'th factor and $\Id_{\E}$ the sheaf morphism induced by the identity of $\E$. 
This is clearly a holomorphic deformation of $\theta$ parametrized by $\CPt$, in particular its index is constant over $\CPt$. 
\begin{prop}\label{prop:hypercohomology}
 The hypercohomology groups of degree $0$ and $2$ vanish for all $\zeta \in \CPt$. 
\end{prop}
\begin{proof}
 We have already seen this for $\zeta \in \Ct \setminus \Pt$ (see the discussion after (\ref{eq:elliptic-complex-Dzeta}) and Proposition \ref{prop:L2Dol}). 
 It suffices to prove the claim for $\zeta \in \Pt \cup \{ \infty \}$. 
 The method of the proof below can be used to treat the case of any $\zeta \in \CPt$.
 
 For $\zeta = \infty$, degeneration at level $2$ of the hypercohomology spectral sequence shows that 
 a class in hypercohomology of degree $0$ would be a global holomorphic section of $\E$ annihilated by $-\Id_{\E}/2 \d z$. 
 Now as $\F$ agrees with $\G$ near $\infty$, such a class may not exist. 
 Similarly, for $\zeta = \infty$ the hypercohomology space of degree $2$ is isomorphic to the cohomology 
 of degree $1$ of the cokernel sheaf of $- \Id_{\E} / 2 \d z$. Now as $\F$ agrees with $\G$ near $\infty$, 
 this cokernel sheaf vanishes near $\infty$, hence is a skyscraper sheaf, and its cohomology of degree $1$ is $0$. 
 
 For $\zeta \in \Pt$, the hypercohomology in degree $0$ is isomorphic to the cohomology of degree $0$ of the kernel sheaf of 
 $\theta - \zeta / 2 \Id_{\E}\d z$. However, this latter kernel sheaf is $0$ by Claim \ref{assn:main3}. 
 Similarly, the hypercohomology in degree $2$ is isomorphic to the cohomology of degree $1$ of the cokernel sheaf of 
 $\theta - \zeta / 2 \Id_{\E}\d z$. However, this latter cokernel sheaf is a skyscraper sheaf by 
 Claim \ref{assn:main3}, so its first cohomology vanishes. 
\end{proof}
It follows from the claim and continuity of the index that the dimension of the first hypercohomology spaces of $\theta_{\zeta}$ are 
constant for all $\zeta \in \CPt$. Therefore, this construction gives an extension of the holomorphic vector bundle $\Et$ over $\CPt$. 

Now this extension allows us in particular to compute the rank of $\Vt$. Indeed, for this purpose it is sufficient to compute the 
dimension of the first hypercohomology space of (\ref{eq:theta-extension}) for $\zeta = \infty$, when the morphism specializes to $-\Id_{\E}/2 \d z$. 
By a simple spectral sequence argument, this latter is the cohomology of degree $0$ of the cokernel sheaf 
$$
  \G\otimes K_{\CP}(2\cdot z_0 + z_1+\cdots + z_n) / \F. 
$$
A contemplation of the definitions of $\F$ and $\G$ yields that this is a skyscraper sheaf supported at the singular points $z_i$ for $1\leq i \leq n$. 
Moreover at such a point $z_i$ it is generated by the classes of 
$$
  \e_i^s \frac{\d z}{z-z_i}
$$
for all values of $s$ such that either
\begin{enumerate}
 \item  $\alpha_i^{j(s)} > 0$ or
 \item  $\alpha_i^{j(s)} = 0$ and $\lambda_i^s \neq 0$. 
\end{enumerate}
Indeed, it is easier to see the contrapositive: the only case when $\g_i^s = (z-z_i)\e_i^s$ and $\f_i^s = \e_i^s$ corresponds to having 
$\alpha_i^{j(s)} = 0, \lambda_i^s = 0$ and $k(s) \geq -1$; however this latter inequality follows from the two equalities and Assumption \ref{assn:main} (\ref{assn:main1}), 
so it can be lifted. In particular, introducing the notation $(\Gr_i^0)_0$ for the kernel of $\res_{z_i}(\theta )^0$ on $\Gr_i^0$, the rank of $\Vt$ is 
\begin{equation}\label{eq:transformed-rank}
 \widehat{r} = \rank (\Vt ) = \sum_{i=1}^n \left ( r - \delta_{0, \alpha_i^0} \dim (\Gr_i^0)_0 \right), 
\end{equation}
where $\delta$ is Kronecker's function. 

We now turn our attention to transforming the parabolic structure. 
For this purpose, let us first recall the notion of an $\R$-parabolic sheaf on a complex manifold $X$ 
with divisor $D_{\red}$, a reduced effective Weil divisor on $X$: 
this is a decreasing family $S_{\bullet}$ of coherent sheaves on $X$ indexed by $\R$ so that 
for all $\alpha\in\R$ 
\begin{enumerate}
 \item left-continuity: there exists some $\varepsilon >0$ with $S_{\alpha - \varepsilon} = S_{\alpha}$; \label{R-par-sheaf-1}
 \item we have $S_{\alpha + 1} = S_{\alpha} \otimes \O_X(-D_{\red})$. \label{R-par-sheaf-2}
\end{enumerate}

Then we have the following result (Section 1, \cite{Yok}):
\begin{clm}\label{clm:parabolic}
The categories of $\R$-parabolic locally free sheaves on $X$ with divisor 
$D_{\red}$ and of parabolic bundles on $X$ with divisor $D_{\red}$ are isomorphic.
\end{clm}

For convenience, let us spell out the correspondence in the case $X=C$ a smooth projective curve. 
To a parabolic bundle $\E$ with filtration (\ref{eq:parfiltr}) of $V_{z_i} = \E_{z_i}$ associate the $\R$-parabolic sheaf 
$\E_{\bullet}$ defined as follows. Near the generic point $z\notin D_{\red}$ for every $\alpha\in\R$ we let $\E_{\alpha} = \E$. 
For $i\in\{ 1, \ldots, n \}$ and $\alpha_i^{l-1} < \alpha \leq \alpha_i^{l}$ we define $\E_{\alpha}$ 
in a small neighborhood of $p_i\in D$ not containing any other $p_{i'}$ as the kernel of the composition map 
$$
  \pi_i^l: \E \xrightarrow{eval_{p_i}} \E_{p_i} = F_i^0 \to F_i^0/F_i^{l}. 
$$
By definition, $\E_{\alpha}$ is then a subsheaf of $\E_0$, which is locally free, hence torsion-free. 
It then follows that $\E_{\alpha}$ is torsion-free, and since $C$ is smooth, $\E_{\alpha}$ is locally free. 
For $\alpha\notin [0,1]$ we extend the definition by property (\ref{R-par-sheaf-2}) above. 

Conversely, to an $\R$-parabolic sheaf we associate the vector bundle whose local sections are given by the sheaf $\E_0$, 
with filtration (\ref{eq:parfiltr}) defined as follows: for any vector $v\in \E_{p_i}$ we let 
$v\in F_i^{l}$ if and only if any local section of $\E_0$ extending $v$ (i.e., whose specialization at 
$p_i$ is $v$) is actually a section of $\E_{\alpha_i^l}$. 
This is clearly the inverse of the construction of the previous paragraph.

After this preliminary correspondence, let us set 
\begin{equation}\label{eq:D-red}
 D_{\red} = z_0 + z_1 + \cdots + z_n
\end{equation}
and construct the parabolic structure of $\Et$. 
First, let us modify the definition of $\F$ and $\G$ from Section \ref{sec:Dolbeault} to take into account a parameter $\alpha\in [0,1)$. 
This goes as follows: for any $1 \leq i \leq n$ 
\begin{enumerate}
 \item \label{case:deletion}
 first, for the values $1\leq s \leq r$ such that $\alpha_i^{j(s)} = 0 = \lambda_i^s$ and \emph{any} $\alpha\in\R$ we set 
  $$
    \f_i^s (\alpha ) = \f_i^s, \quad \g_i^s (\alpha ) = \g_i^s;
 $$
 \item for the values $1\leq s \leq r$ such that at least one of $\alpha_i^{j(s)}, \lambda_i^s$ is non-zero and $\alpha \leq \alpha_i^{j(s)}$ we set \label{case:alpha-smaller}
 $$
    \f_i^s (\alpha ) = \f_i^s, \quad \g_i^s (\alpha ) = \g_i^s;
 $$
 \item for the values $1\leq s \leq r$ such that at least one of $\alpha_i^{j(s)}, \lambda_i^s$ is non-zero and $\alpha_i^{j(s)} < \alpha$ we set \label{case:alpha-greater}
 $$
    \f_i^s (\alpha ) = (z-z_i)\f_i^s, \quad \g_i^s (\alpha ) = (z-z_i)\g_i^s. 
 $$
\end{enumerate}
In the case $i=0$ the definitions are the same, up to replacing the local coordinate $z-z_i$ appearing in the third case by $z^{-1}$. 
\begin{rk}\label{rem:deletion}
 By virtue of Assumption \ref{assn:main} (\ref{assn:main0}), case (\ref{case:deletion}) above does not apply to $i = 0$. 
 Furthermore, case (\ref{case:deletion}) above corresponds to the operation of Deletion of a subscheme from the parabolic divisor of the spectral sheaf in Section 6 of \cite{A-Sz}. 
 Specifically, denote by $\pi_1^* D_{\red}$ the restriction to the spectral curve of the pull-back of the parabolic divisor (\ref{eq:D-red})
 by $\pi_1: \CP \times \CPt \to \CP$. Then we delete the subscheme of $\pi_1^* D_{\red}$ supported away from the section at infinity $\CP \times \{ \infty \}$. 
 Without this deletion, the parabolic divisor of the transformed holomorphic bundle would acquire unwanted additional points of support, placed somewhat 
 arbitrarily on $\Ct \setminus \Pt$, corresponding to the eigenvalues of the restriction of $\theta$ to the space of vectors with $\alpha_i^{j(s)} = 0 = \lambda_i^s$. 
\end{rk}

Cases (\ref{case:alpha-smaller}), (\ref{case:alpha-greater}) may be rephrased as follows: for the unique $m\in \{ 0,1 \}$ satisfying 
\begin{equation}\label{eq:alpha-m}
   \alpha_i^{j(s)} + m -1 < \alpha \leq \alpha_i^{j(s)} + m
\end{equation}
we set 
$$
    \f_i^s (\alpha ) = (z-z_i)^m\f_i^s, \quad \g_i^s (\alpha ) = (z-z_i)^m\g_i^s. 
$$
\begin{comment}
We now extend the definitions in cases (\ref{case:alpha-smaller}), (\ref{case:alpha-greater}) to arbitrary $\alpha \in \R$ as follows: for the unique $m\in \Z$ 
satisfying 
\begin{equation}\label{eq:alpha-m}
   \alpha_i^{j(s)} + m -1 < \alpha \leq \alpha_i^{j(s)} + m
\end{equation}
we set 
$$
    \f_i^s (\alpha ) = (z-z_i)^m\f_i^s, \quad \g_i^s (\alpha ) = (z-z_i)^m\g_i^s. 
$$
%(Observe however that we do not modify case (\ref{case:deletion}).)
\end{comment}
Define $\F_{\alpha}$ and $\G_{\alpha}$ as the locally free sheaves of $\O_{\CP}$-modules isomorphic to $\E$ away from the points $z_i$ and locally generated by the sections 
$\f_i^s (\alpha )$ and $\g_i^s (\alpha )$ respectively, for all $1\leq s \leq r$. 
With these definitions $\F_{\alpha}, \G_{\alpha}$ clearly form a left-continuous, decreasing sequence of locally free sheaves 
parametrized by $[0,1)$. 
%Observe however that they are not $\R$-parabolic sheaves because the conditions 
%$\F_{\alpha + 1} = \F_{\alpha}(-D_{\red}), \G_{\alpha + 1} = \G_{\alpha}(-D_{\red})$ fail to hold as case (\ref{case:deletion}) 
%does not depend on $\alpha$. 
\begin{prop}
 With the above definitions $\theta$ induces for all $\alpha \in [0,1)$ a sheaf morphism 
$$
 \theta_{\alpha} : \F_{\alpha} \to \G_{\alpha} \otimes K_{\CP}(\log(z_1)+\cdots +\log(z_n) + 2\cdot z_0).
$$
\end{prop}
\begin{proof}
 The question is local near the points $z_i$. Let us first treat the case $i\neq 0$. 
 Using the trivialization defined in the proof of Claim \ref{clm:-theta-block-diagonal} we may write 
\begin{equation}\label{eq:theta-on-eis}
   \theta (\e_i^s ) =  \frac{\lambda_i^s \d z}{z-z_i} \e_i^s + \sum_{s'} \vartheta_{ss'}(z) \e_i^{s'}
\end{equation}
where the sum ranges over $1\leq s' \leq r$ such that $\e_i^{s'}$ is an eigenvector of $\res_{z_i}(\theta)$ for the same eigenvalue $\lambda_i^s$ as $\e_i^s$ 
and $\vartheta_{ss'}$ are meromorphic $1$-forms with at most simple poles. Let us distinguish four cases, and use the notation (\ref{eq:alpha-m}).

First, if $\alpha_i^{j(s)} = 0 = \lambda_i^s$, then $\f_i^s = \e_i^s$ and by Assumption \ref{assn:main} (\ref{assn:main2}) all the indices $s'$ in the sum in 
(\ref{eq:theta-on-eis}) also satisfy $\alpha_i^{j(s')}=0$. In particular, for all such indices $s'$ we have $\g_i^s = \e_i^s$ and 
\begin{align*}
    \theta (\f_i^s (\alpha )) & =  \theta (\f_i^s) \\
    & = \theta (\e_i^s ) \\
    & = \frac{\lambda_i^s \d z}{z-z_i} \e_i^s + \sum_{s'} \vartheta_{ss'}(z) \e_i^{s'} \\
    & = \frac{\lambda_i^s \d z}{z-z_i} \g_i^s + \sum_{s'} \vartheta_{ss'}(z) \g_i^{s'}, 
\end{align*}
which is clearly in $\G_{\alpha}\otimes K(z_i)$. 

Second, if $\alpha_i^{j(s)} = 0 \neq \lambda_i^s$ and $k(s)< -1$ then again $\f_i^s = \e_i^s$ and 
$$
  \theta (\f_i^s (\alpha )) = (z-z_i)^m \frac{\lambda_i^s \d z}{z-z_i} \e_i^s + (z-z_i)^m \sum_{s'} \vartheta_{ss'}(z) \e_i^{s'}. 
$$
Now there are two possibilities: either $\e_i^{s'} = \g_i^{s'}$ or $\e_i^{s'} = (z-z_i)^{-1} \g_i^{s'}$. In the first case the $(z-z_i)$-adic valuation of the coefficient of $\g_i^{s'}$ 
on the right-hand side is at least $m-1$, hence the $(z-z_i)$-adic valuation of the coefficient of $\g_i^{s'}(\alpha )$ is at least $-1$. 
The other possibility occurs if and only if $\alpha_i^{j(s')} = 0, k(s')\geq -1$.  Then we necessarily have $j(s) = 0 = j(s')$. 
On the other hand, as $k(s)< -1$ and $N_i^0$ lowers the index of the weight-filtration, the coefficient 
$\vartheta_{ss'}(z)$ actually has no pole at $z_i$, and it again follows that the $(z-z_i)$-adic valuation of the coefficient of $\g_i^{s'}$ 
on the right-hand side is at least $m-1$. 

Third, if $\alpha_i^{j(s)} = 0 \neq \lambda_i^s$ and $k(s)\geq -1$ then $\f_i^s = (z-z_i) \e_i^s$ so 
$$
  \theta (\f_i^s (\alpha )) = (z-z_i)^{m+1} \frac{\lambda_i^s \d z}{z-z_i} \e_i^s + (z-z_i)^{m+1} \sum_{s'} \vartheta_{ss'}(z) \e_i^{s'} 
$$
and the coefficient of $\g_i^{s'}(\alpha )$ in each term on the right-hand side has valuation at least $-1$. 

Finally, if $\alpha_i^{j(s)} > 0$ then by Assumption \ref{assn:main} (\ref{assn:main2}) all the indices $s'$ in the sum in 
(\ref{eq:theta-on-eis}) also satisfy $\alpha_i^{j(s')}>0$, and we conclude as in the first three cases. 

The case $i=0$ can be treated similarly with respect to a local trivialization $\e_i^s$ satisfying (in addition to the 
property of Claim \ref{clm:-theta-block-diagonal}) that every $\e_i^s$ is an eigenvector of $A$. 
\end{proof}

Now, in the same way as for $\theta$, we may consider the twisted Dolbeault complex 
\begin{equation}\label{eq:theta-alpha-zeta}
  \theta_{\alpha, \zeta} = \theta_{\alpha} \otimes s_{\infty} - \frac{\zeta}{2}\Id_{\E}\d z \otimes s_0  : \pi_1^*\F_{\alpha} \to \pi_1^*\G_{\alpha}\otimes K_{\CP}(2\cdot z_0 + z_1+\cdots + z_n) \otimes \pi_2^*  \O_{\CPt}(1)
\end{equation}
where $\Id_{\E}$ is the morphism induced by the identity of $\E$. For all $\alpha \in [0,1)$ let us define 
\begin{equation}\label{eq:transformed-parabolic}
  \Et_{\alpha, \zeta} = \H^1(\theta_{\alpha, \zeta}).
\end{equation}
This will be the fiber over $\zeta$ of a sheaf $\Et_{\alpha}$. 
It follows from the decreasing property that for any $\alpha' \geq \alpha$ there exists a sequence of complexes 
$$
  \theta_{\alpha'} \to \theta_{\alpha}, 
$$
hence the snake lemma implies a sheaf inclusion 
\begin{equation}\label{eq:decreasing}
  \Et_{\alpha'} \subset \Et_{\alpha}  .
\end{equation}
Moreover, the family $\Et_{\bullet}$ is clearly left-continuous. 
\begin{clm}
For all $\alpha \in [0,1)$ the sheaf $\Et_{\alpha}$ is a locally free subsheaf of $\Et_0$ of full rank (\ref{eq:transformed-rank}). 
\end{clm}

\begin{proof}
 The sheaves $\Et_{\alpha}$ are clearly $\O_{\CPt}$-coherent. 
 Left-continuity of $\Et_{\bullet}$ follows immediately from left-continuity of $\F_{\bullet}$ and of $\G_{\bullet}$. 
 The dimension count of the fiber of $\Et_{\alpha}$ at $\infty \in \CPt$ can be carried out similarly to the case of $\Et$, 
 and since we modify $\g_i^s(\alpha )$ in the same way as we do for $\f_i^s(\alpha )$, the computation gives the same value 
 (\ref{eq:transformed-rank}). Furthermore, the same argument as in Proposition \ref{prop:hypercohomology} shows that the dimension 
 of the fiber over $\zeta$ is independent of $\zeta\in \Ct \setminus \Pt$.  For $\alpha > 0$, the sheaves $\Et_{\alpha}$ are 
 subsheaves of the locally free sheaf $\Et$, in particular they are torsion-free. Since $\CPt$ is non-singular, this implies 
 that $\Et_{\alpha}$ is locally free too. 
\end{proof}

Next, let us set
\begin{equation}
 \widehat{D}_{\red} = \mbox{div} (\Pt ) + \infty ,
\end{equation}
where $\mbox{div}$ stands to denote the simple effective divisor associated to a set of distinct points. 
Let us extend the definition of $\Et_{\alpha}$ to all $\alpha \in \R$ by the requirement 
$$
 \Et_{\alpha + 1} = \Et_{\alpha} (-\widehat{D}_{\red} ). 
$$
\begin{prop}
Then, the family $\{ \Et_{\alpha} \}_{\alpha \in \R}$ is an $\R$-parabolic sheaf with divisor $\widehat{D}_{\red}$. 
\end{prop}

\begin{rk}
 It would seem appealing to define the transformed $\R$-parabolic bundle by (\ref{eq:transformed-parabolic}) for all $\alpha \in \R$; 
 however, this definition would not lead to a parabolic sheaf for the following reason. Some branches of $\Sigma$ near the points of $\Pt$ 
 do not intersect $\pi_1^{-1} (\infty )$, and the local subbundle of $\Et$ corresponding to hypercohomology classes supported on such 
 branches do not get twisted by $z^{-1}$ as we increase $\alpha$ by $1$. In the above definition we added manually the twist on this subbundle too. 
 This is the same procedure as Addition of a subscheme to the parabolic divisor in Section 6 of \cite{A-Sz}, an inverse 
 operation of Deletion of Remark \ref{rem:deletion}. 
\end{rk}
\begin{proof}
 We merely need to show the decreasing property (\ref{eq:decreasing}). 
 It is sufficient to prove it for $\alpha' = 1 > \alpha \geq 0$. 
 The proof follows the ideas of Theorem 8.5 \cite{A-Sz}. 
 Define the spectral sheaf $M_{\alpha}$ on $\CP \times \CPt$ as the cokernel sheaf of (\ref{eq:theta-alpha-zeta}) with $\zeta\in \CPt$ varying 
% $$
%  \theta_{\alpha} - \frac{\zeta}{2}\Id_{\E}\d z : \F_{\alpha} \to \G_{\alpha} \otimes K_{\CP}(\log(z_1)+\cdots +\log(z_n) + 2\cdot z_0), 
% $$
 and the spectral curve $\Sigma \subset \CP \times \CPt$ as the support of $M_{\alpha}$ (it is easy to see that the support does not depend on $\alpha$). 
 Then, by the spectral sequence argument of Proposition 4.22 \cite{Sz-these}, 
 $\Et_{\alpha}$ can be defined as the direct image 
 $$
  \Et_{\alpha} = (\pi_2)_* M_{\alpha}
 $$
 of $M_{\alpha}$ by $\pi_2$. %, and $\tt$ is induced by multiplication by $z$. 
 %Now we have 
 %$$
 % \f_i^s (\alpha + 1 ) = (z-z_i) \f_i^s (\alpha ) , \quad \g_i^s (\alpha + 1 ) = (z-z_i) \g_i^s (\alpha )
 %$$
 %for all $1 \leq s \leq r$ except for the ones with $\alpha_i^{j(s)} = 0 = \lambda_i^s$. 
 It follows from Assumption \ref{assn:main} (\ref{assn:main2}) and a simple local computation near $z_i$ that the branches of $\Sigma$ corresponding 
 to vectors not satifying $\alpha_i^{j(s)} = 0 = \lambda_i^s$ are precisely the ones passing through the point $(z_i,\infty) \in \CP \times \CPt$. 
 In addition, as in Claim 4.24 \cite{Sz-these}, up to higher order terms along these branches of $\Sigma$ we have 
 \begin{equation}\label{eq:zeta-z-relationship}
    \zeta = \frac{\lambda_i^s}{z-z_i}.
 \end{equation}
 Thus, multiplying $\f_i^s, \g_i^s$ by $(z-z_i)$ for $\alpha > \alpha_i^{j(s)}$ amounts (up to a scalar) to multiplying them by $\zeta^{-1}$, 
 which is a local defining equation for the point at infinity $\infty \in \CPt$. 
 Therefore, the hypercohomology classes in $\Et_{\alpha}$ defined by representatives locally expressible by some holomorphic linear 
 combination of vectors of the form $(z-z_i)\g_i^s$ not satifying $\alpha_i^{j(s)} = 0 = \lambda_i^s$ are basically 
 hypercohomology classes in $\Et$ defined by representatives locally expressible by some holomorphic linear 
 combination of vectors of the form $\g_i^s$ of the same type, multiplied by $\zeta^{-1}$. 
 Said differently, the subbundle of $\Et_{\alpha}$ near $\infty\in\CPt$ generated by such classes is a subbundle of $\Et(-\infty)$. 
 On the other hand, the subbundle of $\Et_{\alpha}$ generated by hypercohomology classes represented by a linear combination of vectors 
 $\g_i^s(\alpha)$ with $\alpha_i^{j(s)} = 0 = \lambda_i^s$ is clearly the same as the corresponding subbundle of $\Et$. 
 In both cases, the corresponding subbundle of $\Et_{\alpha}$ contains (or is possibly equal to) that of $\Et(-\infty)$. 
 
 A similar argument works for the points $z_i$ with $i \neq 0$. 
\end{proof}

\section{Transformed Higgs field}\label{sec:st}

In this section, we define the transformed Higgs field $\tt$ and show that it preserves the parabolic structure. 
Recall the definition made in (\ref{eq:transformed-parabolic}) and use the convention $\Et = \Et_0$.  

\begin{defn}
The transformed Higgs field $\tt$ is the morphism 
$$
  \Vt \to \Vt \otimes K_{\CPt} (*\widehat{D}_{\red} )
$$
induced by multiplication by $-z/2 \d \zeta$. 
(Here $*\widehat{D}_{\red}$ means that the morphism may have poles of some unspecified order at the support of $\widehat{D}_{\red}$.)
\end{defn}

\begin{thm}\label{thm:main}
\begin{enumerate}
 \item  \label{thm:main1}
 The transformed Higgs field $\tt$ is meromorphic, has a first-order pole at the points of $\Pt$ and a second-order pole at 
 $\infty \in \CPt$ with semi-simple leading-order term, and no other poles. 
 \item \label{thm:main2}
 The residues of $\tt$ at the parabolic points are compatible with the parabolic structure corresponding to the 
 $\R$-parabolic structure $\Et_{\bullet}$, and the leading-order term of $\tt$ at $\infty \in \CPt$ preserves the parabolic filtration. 
 \item \label{thm:main3}
 The map (\ref{eq:Nahm}) is holomorphic for the Dolbeault complex structures of the moduli spaces. 
\end{enumerate}
\end{thm}

\begin{proof}
 (\ref{thm:main1}) 
 The argument of the proof of Theorem 4.9 \cite{Sz-these} shows that multiplication by $-z/2 \d \zeta$ 
 induces a holomorphic map 
 $$
   \Et \to \Et \otimes K_{\Ct \setminus \Pt}
 $$
 over $\Ct \setminus \Pt$. The argument of the proofs of Theorems 4.30 and 4.31 {\it loc. cit.} shows that it 
 has a first-order pole at $\Pt$ and a second-order pole at $\infty \in \CPt$, with semi-simple leading-order term. 

 For statement (\ref{thm:main2}), it is easy to see that compatibility with the parabolic structure associated to the 
 $\R$-parabolic structure $\Et_{\bullet}$ by Claim \ref{clm:parabolic} means precisely that $\tt$ preserves 
 $\Et_{\alpha}$ for all $\alpha \in \R$. Near $\zeta = \infty$ this can be proved as follows. 
 Recall (\ref{eq:zeta-z-relationship}) that the branches of the spectral curve pass through the points 
 $(z_i,\infty) \in \CP \times \CPt$, and according to the formula (\ref{eq:transformed-rank}) these are all 
 the branches intersecting $\pi_2^{-1} (\infty )$. A local section of $\Et_{\alpha}$ near $\zeta = \infty$ 
 is thus represented in hypercohomology by a linear combination 
 \begin{equation}\label{eq:hypercohomology-class}
    \sum_{i=1}^{n} \sum_{s=1}^r h_i^s (z,\zeta ) \g_i^s (\alpha ) 
 \end{equation}
 for some holomorphic functions $h_i^s$, and its image by $\tt$ is by definition represented by 
 $$
   \left( \sum_{i=1}^{n} \sum_{s=1}^r h_i^s (z,\zeta ) \frac{-z \g_i^s (\alpha )}2 \right) \d \zeta .
 $$
 Now, as $\G_{\alpha}$ is an $\O_{\CP}$-module, the expression in parentheses in this latter formula 
 also defines a section of $\G_{\alpha}$. Therefore, by the definition of $\Et_{\alpha}$, the image by 
 $\tt$ of the class (\ref{eq:hypercohomology-class}) represents a section of $\Et_{\alpha} \otimes K_{\CPt}$. 
 A similar argument works in the case $\zeta \in \Pt$. 
 %Now the fiber $\Et_{\alpha, \zeta}$ over $\zeta \in \C$ is generated by hypercohomology classes of the complex 
 
 Statement (\ref{thm:main3}) is immediate: the constructions of Sections \ref{sec:Dolbeault} and \ref{sec:par} 
 are algebraic with respect to the Dolbeault complex structure of the moduli spaces corresponding to Higgs bundles $(\E, \theta)$. 
\end{proof}

\bibliography{FM34}
\bibliographystyle{plain}

\end{document}